\documentclass[12pt]{article}

\usepackage{amssymb}

\usepackage{amsmath,amsthm}

\makeatletter

\@addtoreset{equation}{section}

\@addtoreset{section}{part}

\makeatother

\newcommand{\bpf}[1][Proof]{{\noindent {\sc #1: }}}
\newcommand{\epf}{{\hfill $\square$}\vspace{.5cm}}

\newtheorem{theorem}{Theorem}[section]
\newtheorem{lemma}[theorem]{Lemma}
\newtheorem{corollary}[theorem]{Corollary}
\newtheorem{proposition}[theorem]{Proposition}

\newtheorem{remark}[theorem]{Remark}

\newcommand{\F}        {{ \mathcal F }}
\newcommand{\X}        {{ \mathbb X }}
\newcommand{\G}        {{ \mathcal G }}
\newcommand{\W}        {{ \mathcal W }}
\newcommand{\A}        {{ \mathcal A }}
\newcommand{\Ss}        {{ \mathcal S }}

\newcommand{\R}        {{{\rm I\! R}}}

\newcommand{\E}        {{{\rm I\! E}}}

\renewcommand{\P}        {{ {\rm I \! P}}}
\newcommand{\eps}    {\varepsilon}

\renewcommand{\c}  {\overline{c}}

\def \e    {\varepsilon}
\def \g  {\gamma}

\begin{document}

\title{Homogenization of a singular random one dimensional parabolic PDE
with time varying coefficients}

\author{\'Etienne Pardoux, Andrey Piatnitski}


\maketitle



\begin{abstract}
The paper studies homogenization problem for a non-autonomous parabolic equation with a large random rapidly oscillating potential in the case of one dimensional spatial variable. We show that if the potential is a statistically homogeneous rapidly oscillating function of both temporal and spatial variables then, under proper mixing assumptions, the limit equation is deterministic and the convergence in probability holds. To the contrary, for the potential having a microstructure 
only in one of these variables, the limit problem is stochastic and we only prove
the convergence in law.

\end{abstract}

\section{Introduction}

Our goal is to study the limit, as $\eps\to0$, of the solution of the linear parabolic PDE
\begin{equation}\label{eq-eps}
\left\{
\begin{aligned}
\frac{\partial u^\eps}{\partial t}(t,x)&
=\frac{1}{2}\frac{\partial^2 u^\eps}{\partial x^2}(t,x)
+\eps^{-\gamma} c\left(\frac{t}{\eps^\alpha},\frac{x}{\eps^\beta}\right)u^\eps(t,x),
\quad t\ge0,\, x\in\R;\\
u^\eps(0,x)&=g(x),\quad x\in\R,
\end{aligned}
\right.
\end{equation}
where $g\in L^2(\R)\cap C(\R)$,
$\{c(t,x),\, t\in\R_+,\, x\in\R\}$ is a stationary random field
defined on a probability space $(\Ss,\A,P)$, such that
\begin{equation}\label{c-centered}
E \, c(t,x)=0,\quad t\in\R_+,\, x\in\R,
\end{equation}
where $E$ denotes expectation with respect to the probability measure $P$.
In all this paper, we will assume that the random field $c$ is uniformly bounded, i. e.
$$\sup_{t\ge0,\ x\in\R,\ s\in\Ss}|c(t,x,s)|<\infty.$$
 We define
the correlation function of the random field $c$ as follows~:

\begin{equation}\label{defofphi}
\Phi(t,x):=E\left[ c(s,y)c(s+t,y+x)\right].
\end{equation}

We assume that $\Phi\in L^1(\R\times\R)$. Additional mixing conditions, specific to each particular case, are formulated separately in each section.

We will consider various possible values for the parameters
$\alpha,\, \beta\ge 0$, and we will see that the correct value for
$\gamma$, such that the limit of the highly oscillating term is
non trivial (i.e. finite and non zero), is
\[
\gamma = \left(\frac{\alpha}{4}+\frac{\beta}{2}\right)\vee\frac{\alpha}{2},
\]
and that the highly oscillating term can have three types of limit. If $\alpha=0$, the result is
similar to that obtained in \cite{ifpapi}, that is the limiting PDE is a type of SPDE
driven by a noise which is white in space, and correlated in time. If $\beta=0$, the limit is an
SPDE driven by a noise which is white in time and correlated in space. We believe that in all cases where
$\alpha>0$ and $\beta>0$, the limiting PDE is deterministic. One intuitive explanation of this
result, which was first a surprise for the authors, is the following. In the case $\alpha,\, \beta>0$,
the limiting noise should be white both in time and space, i. e. the limiting PDE should be a
``bilinear'' SPDE driven a space--time white noise. But we know that the corresponding stochastic
integral should be interpreted as a Stratonovich integral, i. e. an It\^o integral plus a correction
term. However, in the space--time white noise case, the correction term is infinite. Hence the correct
choice of $\gamma$ forces the It\^o integral term to vanish, which is necessary for the
``It\^o--Stratonovich correction term''  not to explode.

This result is consistent with that in \cite{GB}

In fact, within the case $\alpha,\ \beta>0$, we have only been able to treat the case where
$0<\beta\le\alpha/2$. The case $0<\alpha<2\beta$ remains open. Our methods do not seem
to cover this last case.

Two variants of the same problem, but with coefficients not depending upon time $t$, have already been considered in \cite{papi_narvik} and in \cite{ifpapi}. The case of random coefficients which are periodic in space was considered in \cite{dioifpapi}.

The paper is organized as follows. In section 2 we state the Feynman--Kac formula for the solution
$u^\eps$ of equation \eqref{eq-eps}. In section 3 we treat the case $\alpha=0$, $\beta>0$. In section 4 we treat the case $0\le 2\beta\le\alpha$, starting with the case $\beta>0$, and finally ending with the case $\beta=0$, $\alpha>0$.

\section{The Feynman--Kac formula}

Let $\{B_t;\, t\ge0\}$ denote a standard Brownian motion
defined on the probability space
$(\Omega,\F,\P)$. The pair $$(\{c(t,x),\ t\ge0, x\in\R\},\{B_t;\, t\ge0\})$$
is defined on the product probability space $(\Omega\times\Ss,\F\otimes\A,\P\times P)$,
so that $\{c(t,x),\ t\ge0, x\in\R\}$ and $\{B_t;\, t\ge0\}$ are mutually independent.

The solution of equation \eqref{eq-eps} is given by the formula
\begin{equation}\label{fk}
\begin{split}
u^\eps(t,x)&=\E\left[g(x+B_t)\exp\left(\eps^{-\gamma}\int_0^t
c\left(\frac{s}{\eps^\alpha},\frac{x+B_s}{\eps^\beta}\right)ds\right)\right]\\
&=\E\left[g(x+B_t)\exp\left(\eps^{-\gamma}\int_0^t\int_\R
c\left(\frac{s}{\eps^\alpha},\frac{x+y}{\eps^\beta}\right)L(ds,y) dy\right)\right],
\end{split}
\end{equation}
where $L(t,x)$ denotes the local time at time $t$ and at level $x$ of the process $B$,
and $\E$ denotes expectation with respect to $\P$. We shall use the notation $X^x_t=x+B_t$.
Note that since $g\in L^2(\R)$ and the density of the law of $X^x_t$ is bounded by $(2\pi t)^{-1/2}$, $g(X_t^x)$ is square integrable.

\section{A criterion for convergence in law}
In the cases where the limit is deterministic, convergence in law is equivalent to convergence in probability. In fact in those cases we will establish convergence in
$L^2(P)$. However, in the case where the limit is random, we are faced with  true
convergence in law. The quantity which should converge in law is a ``partial expectation'', or in other words a conditional expectation. Taking the limit in law of such a quantity does not seem to be very common. In this section, we establish a criterion for convergence in law which is specially tailored for our needs.

\begin{proposition}\label{p_lawconcr}
Let $\{Z^\eps,\,\eps>0\}$ be a collection of real-valued  random variables, and suppose that there exist a random variable $Z$ and, for each $M>0$, random variables $Z_M^\eps$ and $Z_M$ such that
\begin{itemize}
\item[(i)] For any $M$ the sequence $Z_M^\eps$ converges to
$Z_M$ in law, as $\eps\to 0$;
\item[(ii)] It holds
$$
|Z^\eps-Z^\eps_M|\le \frac{\chi^\eps}{M},\qquad |Z-Z_M|\le \frac{\chi^0}{M},
$$
where the family of r. v.'s $\{\chi^\eps,\,\eps\ge0\}$ is tight.
\end{itemize}
Then $Z^\eps$ converges to
$Z$ in law, as $\eps\to 0$.
\end{proposition}
\begin{proof} Since $\{Z^\eps,\ \eps>0\}$ is tight,
it suffices to show that for all $\varphi\in C(\R)$ with $|\varphi(x)|\le 1$ for all $x\in\R$,  and $\varphi$ globally Lipschitz,
$$
E\varphi(Z^\eps)\to E\varphi(Z) \quad\hbox{as }\eps\to0.
$$
Note that
$$
E\varphi(Z^\eps)-E\varphi(Z)=E[\varphi(Z)-\varphi(Z_M)]+
E[\varphi(Z_M^\eps)-\varphi(Z^\eps)]+E\varphi(Z_M) -E\varphi(Z^\eps_M).
$$
If $K$ stands for the Lipschitz constant of $\varphi$, then
$$
\left|E\left\{\varphi(Z)-\varphi(Z_M)+\varphi(Z^\eps_M)- \varphi(Z^\eps)\right\}\right|\le E\inf\left(4,\frac{K}{M}(\chi^\eps+\chi^0)\right).
$$
Consequently, as $M\to\infty$,
$$
\sup_{\eps>0}\left|E\left\{\varphi(Z)-\varphi(Z_M)+\varphi(Z^\eps_M)- \varphi(Z^\eps)\right\}\right|\to0.
$$
The result follows since by (i) for each fixed $M$,  $E\varphi(Z_M)-E\varphi(Z^\eps_M)\to0$, as $\eps\to0$.
\end{proof}

\begin{corollary}\label{cor_crit}
Let $\X$ a Banach space, $\Psi:\Omega\times\X\to\R$ a mapping and $\{W^\eps,\,\eps>0\}$ a family of $\X$--valued random variables defined on $(\Ss,\A,P)$ be such that
\begin{itemize}
\item[(i)] $x\to\Psi(\omega,x)$ is continuous, in $\P$--probability,
\item[(ii)] $\forall x\in\X$, $\omega\to\Psi(\omega,x)$ is $\F$--measurable,
\item[(iii)] for some $\delta>0$, the family
$\{\E|\Psi|^{1+\delta}(\cdot,W^\eps),\,\eps>0\}$  is tight.
\end{itemize}
If moreover $W^\eps$ converges in law towards $W$, then
as $\eps\to0$
$$
\E\Psi(\cdot,W^\eps)\text{ converges in law to }\E\Psi(\cdot,W).
$$
\end{corollary}
\begin{proof}
For $M>0$ and $z\in\R$ write $\psi_M(z)=(z\wedge M)\vee(-M)$.
Note that
$$
\left|\Psi(\cdot,W^\eps)-\psi_M\circ\Psi(\cdot,W^\eps)\right|\le \frac{|\Psi|^{1+\delta}(\cdot,W^\eps)}{M^\delta}.
$$
Consequently, we can apply Proposition \ref{p_lawconcr} with $Z^\eps=\E\Psi(\cdot,W^\eps)$, $Z=\E\Psi(\cdot,W)$, $Z_M^\eps=\E\psi_M\circ\Psi(\cdot,W^\eps)$, $Z_M=\E\psi_M\circ\Psi(\cdot,W)$, since by Lebesgue's dominated convergence theorem
$x\to\E\psi_M\circ\Psi(\cdot,x)$ is continuous from $\X$ into $\R$.
\end{proof}

\begin{remark}
Writing
$$
u^\eps(t,x)=\E\left[g(X^x_t)\exp(Y^{x,\eps}_t)\right]
$$
we will check the third condition of the Corollary with $\delta=1/3$ and we shall use the following H\"older inequality
$$
\E\left[|g|^{4/3}(X_t^x)\exp\Big(\frac{4}{3}Y^{x,\eps}_t\Big)\right]\le \left(\E g^2(X_t^x) \right)^{2/3} \left(\E\exp(4Y_t^{x,\eps}) \right)^{1/3}.
$$
So we have to check that the family $\{\E\exp(4Y_t^{x,\eps}),\,\eps>0\}$ is tight.
\end{remark}

\section{The case $\alpha=0$, $\beta>0$.}
In this case, $\gamma=\beta/2$. Without loss of generality, we restrict ourselves
to the case $\beta=1$. For each $\e>0$, $x\in\R$, we define the process
$$Y^{\eps,x}_t=\frac{1}{\sqrt{\e}}\int_0^t c\left(s,\frac{X^x_s}{\e}\right)ds,\ t\ge0.$$
It will be convenient in this section to assume that for each $x\in \R$,
$t\to c(t,x)$ is a. s. of class $C^2$, and that the $\R^3$--valued random field
\begin{equation}\label{triple_def}
\{(c(t,x),c'(t,x), c''(t,x));\ (t,x)\in\R_+\times\R\},
\end{equation}
is stationary, has zero mean, and is uniformly bounded;
here and later on in this section we use the notation
$$
c'(t,x)=\frac{\partial c}{\partial t}(t,x),\quad c''(t,x)=\frac{\partial^2 c}{\partial t^2}(t,x).
$$

We assume that random field (\ref{triple_def}) is ``$\phi$--mixing in the $x$ direction'', in the sense that the function $\phi:\R_+\to\R_+$ defined by
$$\phi(h)=\sup_{A\in\G_x,\ B\in\G^{x+h},\ P(A)>0}|P(B|A)-P(B)|,$$
where $$\G_x=\sigma\{c(t,z),\ t\ge0,z\le x\}\quad
\G^{y}=\sigma\{c(t,z),\ t\ge0,z\ge y\},$$
satisfies
$$\int_0^\infty \phi^{1/2}(h)dh<\infty.$$
We assume
moreover that (by stationarity, the following quantities do not depend on $t$)
$$\int_{-\infty}^\infty|Ec(t,0)c(t,x)|dx<\infty,\quad
\int_{-\infty}^\infty|Ec'(t,0)c'(t,x)|dx<\infty,$$
$$\int_{-\infty}^\infty|Ec''(t,0)c''(t,x)|dx<\infty.$$

\begin{remark} We suspect that the assumption of $C^2$ regularity is
much stronger than what is necessary for the result that follows to hold. However, in the case of weaker regularity assumptions,
there are technical difficulties which we were not able to overcome.
\end{remark}

\subsection{Weak convergence}
The aim of this subsection is to prove the
\begin{theorem}\label{conv-Y}
For each $t>0$,
\begin{equation}\label{loc_ti_int}
Y^{\eps,x}_t\rightarrow Y^x_t:=\int_0^t\int_\R L(ds,y-x) W(s,dy),
\end{equation}
in $P$--law, as $\eps\to 0$,
where, as above, $L(t,y)$ is the local time
at level $y$ and time $t$ of the Brownian motion $\{X^0_t,\: t\ge0\}$ defined
on
$(\Omega, \F, \P)$, and $\{W(t,y),\: y\in\R\}$ is a centered Gaussian random field defined on
$(\Ss,\A,P)$, with the covariance function
\begin{equation}\label{covwien}
E(W(t,x)W(t',x'))=\begin{cases}
\Psi(t-t')|x|\wedge |x'|,&\ \text{\rm if }\ x\, x'>0;\\
0,& \ \text{\rm if }\ x\,x'<0,
\end{cases}
\end{equation}
where for each $r\in\R$,
\[
\Psi(r)=\int_\R \Phi(r,y)dy,
\]
and the double integral in (\ref{loc_ti_int}) is defined below.
In particular $(X,L)$ and $W$ are independent.
\end{theorem}

We define
\[
W_\eps(t,x)=\frac{1}{\sqrt{\e}}\int_0^xc\left(t,\frac{y}{\eps}\right)dy,
\quad
W'_\eps(t,x)=\frac{1}{\sqrt{\e}}\int_0^xc'\left(t,\frac{y}{\eps}\right)dy.
\]
Note that $\{W_\eps(t,x),W'_\eps(t,x)\}$ is a random field defined on the probability space $(\Ss,\A,P)$.

We first prove
\begin{proposition}\label{conv-W}
The sequence  of random fields $\{(W_\eps, W'_\eps)\}$ converges weakly
as random fields defined on the probability space $(\Ss,\A,P)$, as $\eps \to 0$,
in the space $C(\R_+\times\R;\R^2)$ equipped with the topology of uniform convergence on compact sets, to a centered Gaussian random field
$$\{(W(t,x),W'(t,x)),\ t\ge0, x\in\R\},$$ where the covariance function
of $\{W(t,x)\}$ is given by
\eqref{covwien}, and
$$W'(t,x)=\frac{dW}{dt}(t,x),\quad (t,x)\in \R_+\times\R, \text{ a. s.}$$
\end{proposition}
\bpf
For the sake of clarity of the exposition,
we prove the convergence result for $\{W(t,x)\}$, while the proof
for the pair $\{(W(t,x),W'(t,x))\}$ is essentially identical. The last statement of Proposition \ref{conv-W}
can be obtained by taking the weak limit in the identity
$$W_\eps(t,x)=W_\eps(s,x)+\int_s^tW'_\eps(r,x)dr.$$
We first show that the sequence of random fields $\{W_\eps,\ \eps>0\}$
is tight, as a sequence of random elements of $C(\R_+\times\R)$. Since
$W_\eps(t,0)=0$ for all $t\ge0$, it suffices to estimate the modulus of continuity of $W_\eps$. Now we have
$$|W_\eps(t,x)-W_\eps(s,y)|\le|W_\eps(t,y)-W_\eps(s,y)|+|W_\eps(t,x)-
W_\eps(t,y)|.$$
Concerning the first term, we have
\begin{align*}
\sup_{t\le s\le t+\delta}|W_\eps(s,y)-W_\eps(t,y)|
&\le \int_t^{t+\delta}|W'_\eps(r,y)|dr\\
E\left(\sup_{t\le s\le t+\delta}|W_\eps(s,y)-W_\eps(t,y)|^2\right)&\le
\delta^2 E\left(|W'_\eps(0,y)|^2\right)\\
&\le \delta^2 y\int_\R \left|E[c'(0,0)c'(0,z)]\right| dz.
\end{align*}
It now follows from Chebychev's inequality that
$$\frac{1}{\delta}P\left(\sup_{t\le s\le t+\delta}|W_\eps(s,y)-W_\eps(t,y)|>\eta\right)\le
c(y)\frac{\delta}{\eta^2},$$
from which Billingsley's criteria (8.5) in \cite{Bi} follows.

The increments of $W_\eps$ in the spatial variable can be treated
by an argument very similar to that in the proof of Theorem 20.1 in \cite{Bi}.

Now it remains to identify the limit law of the vector of random processes
$$(W_\eps(t_1,\cdot),\ldots,W_\eps(t_n,\cdot)),$$
for any $n\ge1$, any $0\le t_1<t_2<\cdots<t_n$. It follows from Theorem 20.1 in \cite{Bi}, together with the comments on pages 177 and 178 of that book that the above converges as $\eps\to0$ towards an
$n$--dimensional Wiener process
$$(W(t_1,\cdot),\ldots,W(t_n,\cdot)),$$
which is such that the $(i,j)$ entry of the covariance matrix of the random vector
$(W(t_1,x),\ldots,W(t_n,x))$ is $\Psi(t_i-t_j)|x|$.
\epf

We can now proceed with the

\noindent{\sc Proof of Theorem \ref{conv-Y}} :
We deduce from It\^o's formula that, if
$$\W_\eps(t,x):=\int_0^x W_\eps(t,y)dy,$$
\begin{equation}\label{trickdef}
\begin{array}{rl} \displaystyle
\W_\eps(t,X^x_t)=\!\!&\!\displaystyle \W_\eps(0,x)+\int_0^t\frac{\partial\W_\e}{\partial s}(s,X^x_s)ds\\[5mm]  \displaystyle
+\!\!&\!\displaystyle \int_0^tW_\eps(s,X^x_s)dX^x_s
+\frac{1}{2}\int_0^t\frac{\partial W_\eps}{\partial x}(s,X^x_s)ds,
\end{array}
\end{equation}
consequently
\begin{align}\nonumber
Y^{\eps,x}_t&=\int_0^t\frac{\partial W_\eps}{\partial x}(s,X^x_s)ds\\ \label{intpart}
&=2[\W_\eps(t,X^x_t)-\W_\eps(0,x)-\int_0^t\frac{\partial\W_\e}{\partial s}(s,X^x_s)ds -\int_0^tW_\eps(s,X^x_s)dX^x_s].
\end{align}
The mapping which to $f\in C(\R_+\times\R)$ associates
$g(t,x)=\int_0^{x} f(t,y)dy$
is continuous from $C(\R_+\times\R)$ into itself. Hence it follows from Proposition \ref{conv-W}
that $(W'_\eps,W_\eps,\W_\eps)\Rightarrow(W',W,\W)$ in $C(\R_+\times\R)^3$ as $\eps\to0$,
where $\W(t,x)=\int_0^xW(t,y)dy$, $t\ge0$, $x\in\R$.

Moreover the mappings
$$f\to\int_0^t f(s,X^x_s)dX^x_s,\quad
f\to\int_0^t f(s,X^x_s)ds$$ are continuous from $C(\R_+\times\R)$ into
$L^1(\Omega,\F,\P)$, equipped with the topology of
convergence in probability. Consequently
\begin{equation*}
Y^{\eps,x}_t\to2\Big[\W(t,X^x_t)-\W(0,x)-
\int_0^t\frac{\partial\W}{\partial s}(s,X^x_s)ds
  -\int_0^tW(s,X^x_s)dX^x_s\Big]
\end{equation*}
in $P$ law and $\P$ probability, hence also in $P\times\P$ law.

The result now
follows from the
\begin{lemma}\label{ito}
The following relation holds a. s.
\begin{align*}
\W(t,X^x_t)&=\W(0,x)+\int_0^t\frac{\partial\W}{\partial s}(s,X^x_s)ds+\int_0^tW(s,X^x_s)dX^x_s\\ &\
+\frac{1}{2}\int_0^t\int_\R L(ds,y-x)W(s,dy).
\end{align*}
\end{lemma}
\bpf
Let
\begin{equation}\label{W-mol}
W_n(t,x)=(W(t,\cdot)\ast\rho_n)(x),
\end{equation}
 where $\rho_n(x)=n\rho(nx)$ and $\rho$ is a
smooth map from $\R$ into $\R_+$ with compact support, whose integral over
$\R$ equals one, and $\W_n(t,x)=\int_0^xW_n(t,y)dy$.
Then from It\^o's formula
\begin{align*}
\W_n(t,X^x_t)&=\W_n(0,x)+\int_0^t\frac{\partial\W_n}{\partial s}(s,X^x_s)ds
+\int_0^tW_n(s,X^x_s)dX^x_s\\&\ +\frac{1}{2}
\int_0^t\frac{\partial W_n}{\partial x}(s,X^x_s)ds\\
&=\W_n(0,x)+\int_0^t\frac{\partial\W_n}{\partial s}(s,X^x_s)ds
+\int_0^tW_n(X^x_s)dX^x_s\\& \ +
\frac{1}{2}\int_0^t\int_\R L(ds,y-x)\frac{\partial W_n}{\partial x}(s,y)dy.
\end{align*}
The Lemma now follows by taking the limit as $n\to\infty$, provided we take the limit in the last term, which is done in the
\begin{proposition}
There exists a unique linear mapping
\[
 L\to \{\Lambda(L);\ t\ge0,\ x\in\R\}
\]
from the set of jointly
continuous $L$'s
which are increasing with respect to the $t$ variable and have
compact support in the $x$ variable for all $t$, into the set of
centered Gaussian random fields, with the coraviance function
given by
\[
\E(\Lambda_{t,x}\Lambda_{t',x'})=\int_\R dy\int_0^t\int_0^{t'}
\Psi(s-r) L(ds,y-x)L(dr,y-x'),
\]

where
\[
\Lambda_{t,x}(L)=L^2(\P)-\lim_{n\to\infty}\int_0^t\int_\R L(ds,y-x)\frac{\partial W_n}{\partial x}(s,y)dy.
\]
\end{proposition}
\bpf We first need to show that the right hand side of the formula
for the covariance function of the process $\{\Lambda_{t,x}(L),\: t\ge0,\ x\in\R\}$ is
well defined. This follows from the fact that
\begin{align*}
0&\le\int_\R dy\int_0^t\int_0^{t'}|\Psi(s-r)| L(ds,y-x)L(dr,y-x')\\
 &\le \Psi(0)\int_\R L(t,y-x)L(t',y-x') dy\\
 &<\infty.
\end{align*}
The last inequality follows from the fact that both $L(t,\cdot)$ and
$L(t',\cdot)$
are continuous and have compact support.

Now define
\begin{align*}
\Lambda^{(n)}_{t,x}(L)&=\int_\R dy
\int_0^t\frac{\partial W_n}{\partial y}(s,y)L(ds,y-x)\\
&=\int_\R\int_\R\rho'_n(y-z)\left(\int_0^tW(s,z)L(ds,y-x)\right)dydz.
\end{align*}
In order to complete the proof of the Proposition, it suffices to show that
\begin{equation}\label{conv}
E\left[\Lambda^{(n)}_{t,x}(L)\Lambda^{(m)}_{t',x'}(L)\right]\to
\int_\R dy\int_0^t\int_0^{t'}\Psi(s-r) L(ds,y-x)L(dr,y-x'),
\end{equation}
as $n,m\to\infty$.
Let
$$\Sigma_{n,m}(y,y')=\int_\R\int_\R{\bf1}_{\{zz'>0\}}|z|\wedge|z'|\rho'_n(y-z)
\rho'_m(y'-z')dzdz'.$$
We have
$$E\left[\Lambda^{(n)}_{t,x}(L)\Lambda^{(m)}_{t',x'}(L)\right]=
\int_\R\int_\R\Sigma_{n,m}(y,y')dydy'
\int_0^t\int_0^{t'}\Psi(s-s')L(ds,y-x)L(ds',y'-x').$$
Now an elementary computation based on integration by parts yields
$$\Sigma_{n,m}(y,y')=\int_\R\rho_n(y-z)\rho_m(y'-z)dz,$$
and \eqref{conv} follows from this and the last identity. \epf

We now turn to the case where $L(t,x)$ is the local time of the standard Brownian motion $\{X_t,\:t\ge0\}$, defined on the
probability space $({\Omega},{\F},\P)$.
Thus we now define the stochastic process $\{\Lambda_{t,x}(L),\: t\ge0,\ x\in\R\}$
on the product probability space $(\Ss\times\Omega,\A\otimes\F,P\times\P)$, and denote
$\bar{\P}:=P\times\P$.
We have the
\begin{proposition}
For each fixed $x\in\R$, the process $\{\Lambda_{t,x}(L),\: t\ge0\}$ has a $\bar{\P}$ a. s. continuous modification.
\end{proposition}
\bpf  We have, for $0\le s<t$,
\begin{align*}
\bar{\E}\left(|\Lambda_{t,x}(L)-\Lambda_{s,x}(L)|^p\right)
&=\E\left(\Big|\int_\R dy\int_s^t\int_s^t\Psi(r-r')L(dr,y)L(dr',y)\Big|^{p/2}\right)\\
&\le \Psi(0)^{p/2}\E\left(\Big|\int_\R (L(t,y)-L(s,y))^2 dy \Big|^{p/2}\right)\\
&\le \Psi(0)^{p/2}\E\left(\left|\sup_y(L(t,y)-L(s,y))(t-s)\right|^{p/2}\right),
\end{align*}
where we have used the following well known formula
\[
\int_\R L(t,x)dx=t.
\]
Now from (III) page 200 of Barlow, Yor \cite{BY}, there exists a universal
constant $c_p$ such that
\[
\E\left(\sup_x(L(t,x)-L(s,x))^{p/2}\right)\le
c_p\E\left(\sup_{0\le s\le t}|X_s|^{p/2}\right).
\]
The above right hand side is finite, and
\[
\bar{\E}\left(|\Lambda_t(L)-\Lambda_s(L)|^p\right)
\le C_p(t-s)^{p/2},
\]
from which the result follows, if we choose $p>2$.

\subsection{Convergence of the sequence $u^\eps$}
In order to deduce the convergence of $u^\eps$ from that of $Y^{\e,x}_t$ and
Corollary \ref{cor_crit},
we need some uniform integrability under $\P$
of the collection of random variables
$$\left\{\exp\left[\frac{1}{\sqrt{\e}}\int_0^t c\left(s,\frac{X^x_s}{\e}\right)ds\right],\
\e>0\right\}.
$$
For each $0<\g<1/2$, $t>0$, $\e>0$, we define the $\R_+$--valued random variables
\begin{equation*}
\xi^\e_{t,\g}=\sup_{0\le s\le t,\ x\in\R}\frac{\left|W_\e(s,x)\right|}{(1+|x|)^{1-\g}},\quad
\eta^\e_{t,\g}=\sup_{0\le s\le t,\ x\in\R}\frac{
\left|\frac{\partial W_\e}{\partial s}(s,x)\right|}{(1+|x|)^{1-\g}}.
\end{equation*}
We now prove the
\begin{lemma}\label{le-tight} For each $t>0$, $0<\g<1/2$ and $\e_0>0$, the two collections of random variables
$\{\xi^\e_{t,\g},\ 0<\e\le\e_0\}$ and $\{\eta^\e_{t,\g},\ 0<\e\le\e_0\}$
are tight.
\end{lemma}
\bpf We have
$$\eta^\eps_{t,\gamma}\le\sup_{x\in\R}\frac{
\left|\frac{\partial W_\e}{\partial s}(0,x)\right|}{(1+|x|)^{1-\g}}
+
\sup_{\ x\in\R}\int_0^t\frac{
\left|\frac{\partial^2 W_\e}{\partial s^2}(s,x)\right|}{(1+|x|)^{1-\g}}ds,$$
and similarly
$$\xi^\eps_{t,\gamma}\le\sup_{x\in\R}\frac{
\left|W_\e(0,x)\right|}{(1+|x|)^{1-\g}}
+
\sup_{\ x\in\R}\int_0^t\frac{
\left|\frac{\partial W_\e}{\partial s}(s,x)\right|}{(1+|x|)^{1-\g}}ds.$$
It remains to show that each of the four collections of r. v. appearing in the two above right hand sides is tight. Each of the four terms can be treated by the eexact same argument as used in the proof of Lemma 5 page 295--296 of
\cite{papi_narvik}, which we now reproduce for the convenience of the reader, in the case of the first term of the second right--hand side.
\epf

We drop the index $t$ for simplicity, and define
\[
\zeta^\eps_\gamma=\sup_{x\in\R}\frac{|W_\eps(x)|}{(1+|x|)^{1-\gamma}}.\\
\]
We have the
\begin{lemma}\label{le-estim}
For any $0<\gamma<1/2$ and $\eps_0>0$, the collection of random variables
$\{\zeta^\eps_\gamma,\: 0<\eps\le\eps_0\}$
is tight.
\end{lemma}
\bpf
Due to the symmetry it is sufficient to estimate $|W_\eps(x)|$ for
$x>0$. We have
\begin{align*}
E(|W_\eps(r)|^2)&=\eps\int\limits_0^{r/\eps}\int\limits_0^{r/\eps}
E(c(x)c(y))dxdy\\&
\le
2\eps\int\limits_0^{r/\eps}\int\limits_0^\infty|E(c(0)c(x))|dsxdy
\\&
\le 2rc_0.
\end{align*}

Denote by $\mathcal{G}_x=\sigma\{c(y),\ y\le x\}$ and
$$
\eta_x=\int\limits_0^\infty E(c(y+x)|{\mathcal G}_x)dy.
$$

Combining the estimate (2.23) in the case $p=\infty$ in Proposition 7.2.6.
from \cite{EtKu} with our condition that the correlation function $\Phi$ is both bounded and integrable, we deduce that
  the stationary process $\{\eta_x,\: x\ge0\}$ satisfies
 $|\eta_x|\le c_1$ a.s. for all $x>0$, with a non-random constant
$c_1$. Moreover,
$$
\int_0^x c(r)dr-\eta_x
$$
is a square integrable ${\G}_x$ martingale. Denote it by ${\cal
N}_x$. Clearly
\begin{align*}
W_\eps(x)&=\frac{\sqrt{\eps}}{\bar c}\int_0^{x/\eps} c(y)dy\\&
=
\frac{\sqrt{\eps}}{\bar c}{\cal N}_{x/\eps}+
\frac{\sqrt{\eps}}{\bar c}\eta_{x/\eps},
\end{align*}
and thus we deduce from Doob's inequality
\begin{align*}
E\Big(\mathop{\rm sup}\limits_{0\le x\le r} |W_\eps(x)|^2\Big) &\le
\frac{2}{\bar c^2}E(\mathop{\rm sup}\limits_{0\le x\le
r/\eps}\big({\sqrt{\eps}\cal N}_{x}\big)^2)+2\frac{c^2_1\eps}{\c^2}
\\&
\le \frac{4}{\bar c^2}E((\sqrt{\eps}{\cal N}_{r/\eps})^2)+
2\frac{c^2_1\eps}{\c^2} \\		
&\le 8E(|W_\eps(r)|^2)+10\frac{c^2_1\eps}{\c^2}\\
&\le C(\eps+r),    
\end{align*}
provided $C=(16c_0)\vee(10c_1^2/\c^2)$. Now for $j\ge1$, $M>0$,
\begin{align*}
P\left(\sup_{2^{j-1}<r\le 2^j}\frac{|W_\eps(r)|}{(1+r)^{1-\gamma}}\ge M\right)
&\le P\left(\sup_{0\le r\le 2^j}|W_\eps(r)|\ge (1+2^{j-1})^{1-\gamma}M\right)\\
&\le\frac{C(\eps+2^j)}{M^2(1+2^{j-1})^{2-2\gamma}}\\
&\le(\eps\vee1)\frac{2C}{M^2}(1+2^{j-1})^{2\gamma-1}.
\end{align*}
Summing up over $j\ge1$, we deduce that
\begin{align*}
P\left(\zeta^\eps_\gamma\ge M\right)&\le
2P\left(\sup_{r>0}\frac{|W_\eps(r)|}{(1+r)^{1-\gamma}}\ge M\right)\\
&\le(\eps\vee1)\frac{4C}{M^2}\sum_{j=0}^\infty(1+2^{j})^{2\gamma-1}\\
&\le (\eps\vee1)\frac{C'}{M^2}.
\end{align*}
This completes the proof of Lemma.
\epf

We can now establish the required uniform integrability
\begin{proposition}\label{ui}
The collection of random variables
$$\left\{\E\left(\exp\left[\frac{4}{\sqrt{\e}}\int_0^t c\left(s,\frac{X^x_s}{\e}\right)ds\right]\right),\
\e>0\right\}
$$
is $P$--tight.
\end{proposition}
\bpf
We make use of the following easy estimate : if $Z$ is an $N(0,1)$ random variable, $c>0$ and $0<p<2$,
\begin{equation}\label{expgauss}
\E\exp(c|Z|^p)\le\sqrt{2}\exp\left[\frac{2-p}{2}(4c)^{\frac{2}{2-p}}\right].
\end{equation}
>From the identity (\ref{trickdef}) in the proof of Theorem \ref{conv-Y}, we deduce that
\begin{align*}
\frac{4}{\sqrt{\e}}\int_0^t c\left(s,\frac{X^x_s}{\e}\right)ds&=
8\W_\e(t,X^x_t)-8\W_\e(0,x)-8\int_0^t\frac{\partial \W_\e}{\partial s}(s,X^x_s)ds\\&\ -8\int_0^tW_\e(s,X^x_s)dB_s.
\end{align*}
Hence
\begin{align*}
\E\left(\exp\left[\frac{4}{\sqrt{\e}}\int_0^t c\left(s,\frac{X^x_s}{\e}\right)ds\right]\right)
&\le e^{-8\W_\e(0,x)}\left[\E\left(e^{24\W_\e(t,X^x_t)}\right)\right]^{1/3}\\ \times
\left[\E\left(e^{-24\int_0^t\frac{\partial\W_\e}{\partial s}(s,X^x_s)ds}\right)\right]&^{1/3}
\left[\E\left(e^{-24\int_0^t\W_\e(s,X^x_s)dB_s}\right)\right]^{1/3}.
\end{align*}
It remains to dominate each of the 4 factors of the right--hand side of the last identity by a tight sequence, which
we now do, with the help of Lemma \ref{le-tight}. Below $\gamma$ is an arbitrarily fixed number in the interval $(0,1/2)$.
Clearly,
$$
-8\W_\e(0,x)\le8|x|(1+|x|)^{1-\gamma}\xi^\eps_{0,\gamma},
$$
and the sequence on the right-hand side is tight as well as the sequence of the exponentials $\exp\big(8|x|(1+|x|)^{1-\gamma}\xi^\eps_{0,\gamma}\big)$.
Next
\begin{align*}
24\W_\e(t,x+B_t)&=24\int_0^{x+B_t}W_\e(t,y)dy\\
&\le24|x+B_t|(1+|x+B_t|)^{1-\gamma}\xi^\e_{t,\gamma}\\
&\le48[(1+|x|)^{2-\gamma}+|B_t|^{2-\gamma}]\xi^\e_{t,\gamma}.
\end{align*}
Hence from \eqref{expgauss},
$$\E\left(e^{24\W_\e(t,X^x_t)}\right)\le
\sqrt{2}\exp\left[48(1+|x|)^{2-\gamma}\xi^\e_{t,\gamma}\right]
\exp\left[\frac{\gamma}{2}\left(192\xi^\e_{t,\gamma} t^{1-\gamma/2}\right)^{2/\gamma}\right]
.$$
Similarly,
\begin{equation*}
-24\int_0^t\frac{\partial\W_\e}{\partial s}(s,X^x_s)ds\le
24\eta^\e_{t,\gamma}\int_0^t\left(|x+B_s|+\frac{1}{1-\gamma}|x+B_s|^{2-\gamma}\right)ds,
\end{equation*}
so using Jensen's inequality, we get
\begin{equation*}
\exp\left(-24\int_0^t\frac{\partial\W_\e}{\partial s}(s,X^x_s)ds\right)
\le\frac{1}{t}\int_0^t\exp\left[
24t\eta^\e_{t,\gamma}\left(|x+B_s|+\frac{1}{1-\gamma}|x+B_s|^{2-\gamma}\right)\right]ds,
\end{equation*}
from which the result follows as above. Next from Cauchy--Schwarz,
\begin{align*}
\E\exp\left(-24\int_0^t\W_\e(s,X^x_s)dB_s\right)&\le
\left[\E\exp\left(-48\int_0^t\W_\e(s,X^x_s)dB_s-1152 \int_0^tW_\e^2(s,X^x_s)ds\right)\right]^{1/2}\\
&\quad\times\left[\E\exp\left(1152 \int_0^tW_\e^2(s,X^x_s)ds\right)\right]^{1/2}
\\&\le\left[\E\exp \left(1152\int_0^tW_\e^2(s,X^x_s)ds\right)\right]^{1/2},
\end{align*}
but
$$\int_0^tW_\e^2(s,X^x_s)ds\le \left[\xi^\e_{t,\gamma}\right]^2
\int_0^t(1+|x+B_s|)^{2-2\gamma}ds,$$
and we estimate this term again using Jensen's inequality and the
inequality \eqref{expgauss}.
\epf

It now follows from Theorem \ref{conv-Y}, Propositions \ref{conv-W} and \ref{ui}, and the fact that by formula (\ref{intpart}) the exponent in the Feynman--Kac formula is a continuous function of $(W,W'_t)$, that we can apply  Corollary \ref{cor_crit}, yielding
\begin{theorem}\label{conv_alpha=0}
For any $(t,x)\in\R_+\times\R$,
$$u^\eps(t,x)\to u(t,x):=\E\left[g(X^x_t)\exp\left(
\int_0^t\int_\R L(ds,y-x) W(s,dy)\right)\right]$$
in $P$--law, as $\eps\to0$.
\end{theorem}

\begin{remark} Note that it is not clear how the limiting exponent in the Feynman--Kac formula
could be written
in terms of $W$ and $B$.
\end{remark}

The corresponding limiting SPDE reads
\begin{equation*}
\left\{
\begin{aligned}
\frac{\partial u}{\partial t}(t,x)&
=\frac{1}{2}\frac{\partial^2 u}{\partial x^2}(t,x)dt
+u(t,x)\circ W(t,dx),
\quad t\ge0,\, x\in\R;\\
u(0,x)&=g(x),\quad x\in\R,
\end{aligned}
\right.
\end{equation*}
where the stochastic integral should be interpreted as an anticipative Strato\-no\-vich integral, see \cite{nualart}, \cite{nupa}. Since
anticipating stochastic integrals are not very easy to handle, we
prefer to rewrite the above SPDE as follows, using the same trick as in \cite{papi_narvik}. We note that $u(t,x)\circ W(t,dx)$ is a convenient notation for the product
$$u(t,x)\frac{\partial W}{\partial x}(t,x)=\frac{\partial (uW)}{\partial x}(t,x)-\frac{\partial u}{\partial x}(t,x)W(t,x).$$
Hence we rewrite the above SPDE in the form
\begin{equation}
\left\{
\begin{aligned}
\frac{\partial u}{\partial t}(t,x)&
=\frac{1}{2}\frac{\partial^2 u}{\partial x^2}(t,x)
+\frac{\partial (uW)}{\partial x}(t,x)
-\frac{\partial u}{\partial x}(t,x)W(t,x),
\quad t\ge0,\, x\in\R;\\
u(0,x)&=g(x),\quad x\in\R.
\end{aligned}
\right.
\end{equation}

\section{The case $0\le2\beta\le\alpha$, $\alpha>0$}

We first prove two Propositions which will be useful in two of
the three following subcases.

We first recall the definition of the uniform mixing coefficient $\alpha_{\rm um}(r)$
of the random field $c(t,x)$. For a set $A\subset
\R^2$ denote by $\F_A$ the $\sigma$-algebra generated by
$\{c(t,x)\,:\,(t,x)\in A\}$. We set
$$
\alpha_{\rm um}(r)=\sup\limits_{\begin{array}{cc}{\scriptstyle A_1,A_2\subset \R^2}\\
{\scriptstyle\mathrm{dist}(A_1,A_2)\ge r}\end{array}}
\sup\limits_{\begin{array}{cc}{\scriptstyle {\cal S}_1\in\F_{A_1}}\\
{\scriptstyle {\cal S}_2\in\F_{A_2},\ P({\cal S}_2)\not=0}\end{array}}|P({\cal
S}_1|{\cal S}_2)-P({\cal S}_1)|.
$$
Next we recall the definition of the maximum
correlation coefficient $\rho(r)$~:
$$
\rho_{\rm mc}(r)=
\sup\limits_{\begin{array}{cc}{\scriptstyle A_1,A_2\subset R^2}\\
{\scriptstyle\mathrm{dist}(A_1,A_2)\ge r}\end{array}}
\sup\limits_{\xi,\eta}\ |E(\xi\eta)|,
$$
where the second supremum is taken over all $\F_{A_1}$-measurable
$\xi$ and $\F_{A_2}$-measurable $\eta$ such that $E\xi\, E\eta =0$,
$|\xi|\le 1,\ |\eta|\le1$.

We shall assume in this section that there exists $C,\ \delta>0$ such that

\begin{equation*}
{ \alpha_{\rm um}(r)\le C(1+r)^{-(2+\delta)}.}\leqno{(H_{\rm um})}
\end{equation*}

Proposition 7.2.2, page 346 of \cite{EtKu}, with, using the notations there,
$s=\infty$, $r=1$, $p=\infty$ and $q=1$, yields the
\begin{lemma}\label{lemEK}
It follows from $(H_{\rm um})$ that for some constant $C'$,
$$\rho_{\rm mc}(r)\le C'(1+r)^{-(2+\delta)},$$
and in particular $\rho_{\rm mc}\in L^1(\R_+)$.
\end{lemma}

An immediate consequence of the Lemma is the
\begin{corollary}\label{corEK}
There exists a constant $C$ such that for all $t\ge0$, $x\in\R$,
$$|\Phi(t,x)|\le C(1+t+|x|)^{-(2+\delta)}.$$
\end{corollary}

Recall the function $\Phi$ defined in (\ref{defofphi}).
It will be convenient in the sequel to use the fact that
there exists a bounded function
$\Psi~:\R_+\times\R\to\R_+$ such that
$$
|\Phi(s,x)|\le \Psi(s,x),
$$
$x\,\to\,\Psi(s,x)$ is decreasing on $\R_+$ for all
$s\in\R_+$, $\Psi(s,-x)=\Psi(s,x)$, and

\begin{equation}\label{cov_ubou}
\int_0^\infty \Psi(t,0)dt<\infty;\quad
\int_0^\infty \Psi(t,x)dt\,\to\,0, \ \text{ as }|x|\to\infty.
\end{equation}
For example,  we might set (for $x>0$)
$$
\Psi(t,x)=\sup\limits_{\begin{array}{c}\\[-7mm] \scriptstyle s\ge t\\[-2mm]
\scriptstyle |y|\ge x\end{array}} |\Phi(s,y)|.
$$
In this case, (\ref{cov_ubou}) follows from our standing
assumption $(H_{\rm um})$, see Corollary \ref{corEK}.

Whithout loss of generality, we assume that $\alpha=1$. Hence we want to
treat the case $0\le\beta\le1/2$.
The exponent in the Feynman--Kac formula reads
$$Y^\e_t=\frac{1}{\sqrt{\e}}\int_0^t
c\left(\frac{s}{\e},\frac{x+B_s}{\e^\beta}\right)ds.$$
Let us first prove the
\begin{proposition}\label{prop-tcl}
Assume that the condition $(H_{\rm um})$ holds. Then for all $0\le\beta\le1/2$,
the limit relation holds in $\P$--probability
\begin{equation}\label{Pconvlaw}
\lim\limits_{\eps\to 0}E\exp(Y_t^{\eps,x})=\exp(t\Sigma)
\end{equation}
with
\begin{equation}\label{def_sigg}
\Sigma(\beta)=\begin{cases}\displaystyle
\int_{-\infty}^{+\infty}\Phi(u,0)du,&\text{if $0\le\beta<1/2$},\\[4mm]
\displaystyle
\int_{-\infty}^{+\infty}\E\Phi(u,B_u)du,&\text{if $\beta=1/2$}.
\end{cases}
\end{equation}
\end{proposition}
\bpf
We only consider the case $\beta=1/2$, for $\beta\in(0,1/2)$ the desired statement can be justified in the same way with some simplifications.

We introduce a partition of the interval $(0,t/\eps)$ into alternating subintervals
of the form
$$
{\cal I}^\eps_j=\big((\eps^{-1/3}+\eps^{-\nu})j,\ (\eps^{-1/3}+\eps^{-\nu})j+\eps^{-1/3}\big),\quad j=1,2,\dots,K^\eps,
$$
$$
{\cal J}^\eps_j=\big((\eps^{-1/3}+\eps^{-\nu})j+\eps^{-1/3},\ (\eps^{-1/3}+\eps^{-\nu})(j+1)\big),\quad j=1,2,\dots, K^\eps;
$$
here $K^\eps=[(\eps^{-1}t)/(\eps^{-1/3}+\eps^{-\nu})]$, $[\cdot]$ stands for the integer part, and $0<\nu<1/3$. This implies that $K^\eps= t\eps^{-2/3}(1+o(1))$.
Denote
$$
\eta_j^\eps=\sqrt{\eps}\int\limits_{{\cal I}_j^\eps} c\big(s,\frac{x}{\sqrt{\eps}}+\tilde B_s)\,ds,\qquad \zeta_j^\eps=\sqrt{\eps}\int\limits_{{\cal J}_j^\eps} c\big(s,\frac{x}{\sqrt{\eps}}+\tilde B_s)\,ds
$$
where the new Wiener process $\tilde{B}_s$ has been obtained from the original one by the scaling $\sqrt{\eps}\tilde{B}_{s/\eps}=B_s$. We may assume without loss of generality that the process $\tilde B_s$ is fixed.
Then
$$
Y_t^{\eps,x}=\sum\limits_{j=0}^{K^\eps} (\eta_j^\eps+\zeta_j^\eps) +V_\eps,
$$
where $|V_\eps|\le C\eps^{1/3}$ $P\times\mathbb P$-a.s.

Notice that, due to the standing assumptions on $c(s,x)$, there exists a constant $C$ such that
\begin{equation}\label{est_eta-zeta}
|\eta_j^\eps|\le C\eps^{1/6},\qquad |\zeta_j^\eps|\le C\eps^{1/2-\nu}.
\end{equation}

To use efficiently the mixing properties of the coefficients it is convenient to represent $Y_t^{\eps,x}$ as follows
$$
Y_t^{\eps,x}=\sum\limits_{j {\rm\,is\, even}}\eta_j^\eps+\sum\limits_{j {\rm\,is\, odd}}\eta_j^\eps+\sum\limits_{j=0}^{K^\eps}\zeta_j^\eps\ :=Y_e^\eps+Y_o^\eps+{\cal Y}^\eps.
$$
First, let us compute the limit of $E\exp(Y_e^\eps)$. For the sake of definiteness we may assume that $K^\eps$ is odd. The case of even $K^\eps$ can be treated in exactly the same way.
Using the notation ${\cal A}_j^\eps=\sigma\{c(s,x)\,:\, s\le (\eps^{-1/3}+\eps^{-\nu})j,\ x\in \mathbb R\}$, we have
$$
E\exp(Y_e^\eps)=E\exp\Big(\sum\limits_{j=0}^{(K^\eps-1)/2}\eta_{2j}^\eps\Big)
=E\bigg(E\Big\{\exp\Big(\sum\limits_{j=0}^{(K^\eps-1)/2}\eta_{2j}^\eps\Big)\Big|{\cal A}_{(K^\eps-2)}^\eps\Big\}\bigg)
$$
$$
=E\bigg(\exp\Big(\sum\limits_{j=0}^{(K^\eps-3)/2}\eta_{2j}^\eps\Big)E\big\{\exp( \eta_{K^\eps-1}^\eps)\big|{\cal A}_{(K^\eps-2)}^\eps\big\}\bigg)
$$
$$
=E\bigg(\exp\Big(\sum\limits_{j=0}^{(K^\eps-3)/2}\eta_{2j}^\eps\Big) \big[E\exp(\eta_{K^\eps-1}^\eps)+E\big\{ (\exp(\eta_{K^\eps-1}^\eps)-E\exp(\eta_{K^\eps-1}^\eps))\big|{\cal A}_{(K^\eps-2)}^\eps\big\}\big]\bigg).
$$
Since, according to (\ref{est_eta-zeta}), $|\exp(\eta_{K^\eps-1}^\eps)-E\exp(\eta_{K^\eps-1}^\eps)|\le C\eps^{1/6}$, then, by Proposition 2.6 page 349 in \cite{EtKu}, we have
$$
\big|E\big\{ (\exp(\eta_{K^\eps-1}^\eps)-E\exp(\eta_{K^\eps-1}^\eps))\big|{\cal A}_{(K^\eps-2)}^\eps\big\}\big|\le C\alpha_{\rm um}(\eps^{-1/3})\eps^{1/6}\le C\eps^{(2+\delta)/3+1/6}
$$
Combining this estimate with the evident bound $1/2\le E\exp(\eta_{K^\eps-1}^\eps)\le 2$, we obtain
$$
E\exp(Y_e^\eps)=E\bigg(\exp\Big(\sum\limits_{j=0}^{(K^\eps-3)/2}\eta_{2j}^\eps\Big) E\exp(\eta_{K^\eps-1}^\eps)\big(1+O(\eps^{(2+\delta)/3+1/6})\big)\bigg)
$$
with $|O(\eps^{(2+\delta)/3+1/6})|\le C\eps^{(2+\delta)/3+1/6}$.
Iterating this process, after $K^\eps/2$ steps we arrive at the equality
$$
E\exp(Y_e^\eps)=\prod\limits_{j=0}^{(K^\eps-1)/2} E\exp(\eta_{2j}^\eps)\big(1+O(\eps^{(2+\delta)/3+1/6})\big).
$$
Since $\prod\limits_{j=0}^{(K^\eps-1)/2}\big(1+O(\eps^{(2+\delta)/3+1/6})\big)$ converges to $1$ as $\eps\to0$, we have
\begin{equation}\label{aux_pro01}
\lim\limits_{\eps\to0}E\exp(Y_e^\eps)=
\lim\limits_{\eps\to0}\prod\limits_{j=0}^{(K^\eps-1)/2}E\exp(\eta_{2j}^\eps)
\end{equation}
We proceed with estimating the term $E\exp(\eta_j^\eps)$. Using Taylor expansion of the exponent about zero results in the following relation
\begin{equation}\label{aux_pro02}
E\exp(\eta_j^\eps)=1+E\eta_j^\eps+\frac{1}{2}E\big((\eta_j^\eps)^2\big)+ \frac{1}{6}E\big((\eta_j^\eps)^3\big)+ \frac{1}{24}E\big((\eta_j^\eps)^4\big)+O(\eps^{5/6}),
\end{equation}
here we have also used the bound $|\eta_j^\eps|\le C\eps^{1/6}$. By the centering condition on $c(\cdot)$, $E\eta_j^\eps=0$. Considering  $\lambda_j^{\eps}$ defined in the proof of Lemma \ref{l_uni_inte} below in the particular case $\gamma=1/3$, $\nu=0$, we have that  $\eta_j^\eps$ and $\lambda_j^{\sqrt{\eps}}$ have the same law.
It then follows from (\ref{vst_3}) that
\begin{equation}\label{aux_pro03}
\frac{1}{6}\big|E\big((\eta_j^\eps)^3\big)\big|\le C\eps^{5/6},\qquad \frac{1}{24}E\big((\eta_j^\eps)^4\big)\le C\eps.
\end{equation}

The contribution of the term $\frac{1}{2}E\big((\eta_j^\eps)^2\big)$ can be computed as follows
\begin{equation}\label{aux_pro04}
\frac{1}{2}E\big((\eta_j^\eps)^2\big)=\frac{\eps}{2}\int\limits_{{\cal I}^\eps_j} \int\limits_{{\cal I}^\eps_j}E\big\{c\big(r,\frac{x}{\sqrt{\eps}}+\tilde B_r) c\big(s,\frac{x}{\sqrt{\eps}}+\tilde B_s)\big\}\,dsdr
\end{equation}
$$
=\frac{\eps}{2}\int\limits_{{\cal I}^\eps_j} \int\limits_{{\cal I}^\eps_j}\Phi(r-s,\tilde B_r-\tilde B_s)\,drds\,:=\Xi^\eps_j.
$$
By definition and due to the properties of the Wiener process, the random variables $\Xi^\eps_j=\Xi^\eps_j(\omega)$, $j=1,2\dots,K^\eps$, are independent, identically distributed and satisfy
the following bounds
\begin{equation}\label{aux_pro05}
C_0\eps^{2/3}\le \Xi^\eps_j\le C_1\eps^{2/3}, \qquad \mathbb E\Xi^\eps_j=\Sigma(1/2)\eps^{2/3}+O(\eps)
\end{equation}
with $0<C_0<C_1<\infty$ and $|O(\eps)|\le C_2\eps$; the quantity $\Sigma(1/2)$ has been defined in (\ref{def_sigg}). Combining (\ref{aux_pro01})--(\ref{aux_pro05}) yields
\begin{equation}\label{aux_pro06}
\lim\limits_{\eps\to0}E\exp(Y_e^\eps)= \lim\limits_{\eps\to0}\prod\limits_{j=0}^{(K^\eps-1)/2}(1+\Xi^\eps_j+O(\eps^{5/6}))
\end{equation}
$$
=\exp\Big(\lim\limits_{\eps\to0}\sum\limits_{j=0}^{(K^\eps-1)/2}\Xi^\eps_j\Big)
=\exp\big(\lim\limits_{\eps\to0}[(K^\eps/2)\,\mathbb E\Xi^\eps_j]\big)=\exp\Big(\frac{t\Sigma(1/2)}{2}\Big)
$$
in $\mathbb P$ probability, from the weak law of large numbers.
Similarly
\begin{equation}\label{aux_pro07}
\lim\limits_{\eps\to0}E\exp(Y_o^\eps)=\exp\Big(\frac{t\Sigma(1/2)}{2}\Big).
\end{equation}
Exploiting exactly the same arguments one can show that
$$
\lim\limits_{\eps\to0}E\exp({\cal Y}^\eps)=1
$$
in $\mathbb P$ probability.
In view of the strict convexity and the strict positivity for $x\not=0$ of the function
$\varphi(x)=e^x-1-x$, this implies that, as $\eps\to0$,
\begin{equation}\label{aux_pro08}
{\cal Y}^\eps\to0\ \text{ in }\mathbb{P}\times P\text{ probability.}
\end{equation}
Following the line of the proof of estimate (\ref{vst_0}) in Lemma \ref{l_uni_inte}
below, one can show that
$$
E\exp(4Y_{e,o}^\eps)\le C,\qquad E\exp(4{\cal Y}^\eps) \le C
$$
with a deterministic constant $C$. Thanks to these bounds we deduce from (\ref{aux_pro08}) that
\begin{equation}\label{aux_pro09}
\lim\limits_{\eps\to0}E\exp(Y_e^\eps+Y_o^\eps+{\cal Y}^\eps)= \lim\limits_{\eps\to0}E\exp(Y_e^\eps+Y_o^\eps)
\end{equation}
in $\mathbb P$ probability.

Denote
$$
{\cal A}^\eps_e=\sigma\Big\{c(s,x)\,:\,s\in \bigcup\limits_{j=0}^{(K^\eps-1)/2} {\cal I}_{2j}^\eps, \ x\in\mathbb R\Big\},
$$
By construction,
$$
{\rm dist}\Big(\bigcup\limits_{j=0}^{(K^\eps-1)/2} {\cal I}_{2j}^\eps, \bigcup\limits_{j=0}^{(K^\eps-1)/2} {\cal I}_{2j+1}^\eps\Big)=\eps^{-\nu}.
$$
Therefore,
$$
E\exp(Y_e^\eps+Y_o^\eps)=E\big\{\exp(Y_e^\eps)E\big(\exp(Y_o^\eps)|{\cal A}^\eps_e\big)\big\}=
$$
$$
E\{\exp(Y_e^\eps)\}\,E\{\exp(Y_e^\eps)\}+o(\eps^{2\nu})
$$
and
$$
\lim\limits_{\eps\to0}E\exp(Y_t^{\eps,x})= \lim\limits_{\eps\to0}E\{\exp(Y_e^\eps)\}\, \lim\limits_{\eps\to0}E\{\exp(Y_e^\eps)\}=\exp(t\Sigma),
$$
as required.
\epf

Now, consider the process $\exp\big(Y_t^{\eps,x}(\omega)\big) \exp\big(Y_t^{\eps,x}(\omega_1)\big)$ defined on the product space $\Omega\times\Omega$ with the product measure $\mathbb P\times\mathbb P$.

\begin{proposition}\label{prop-tcl-prod}
Assume that the condition $(H_{\rm um})$ holds. Then for all $0\le\beta\le1/2$,
the limit relation holds in $\P\times\P$--probability
\begin{equation}\label{Pconvlaw-prod}
\lim\limits_{\eps\to 0}E\big\{\exp\big(Y_t^{\eps,x}(\omega)\big) \exp\big(Y_t^{\eps,x}(\omega_1)\big)\big\}=\exp(2t\Sigma)
\end{equation}
\end{proposition}
\bpf
It is easy to check that for the standard Brownian motion $B_s$ and for any $t>0$ the limit relation holds
\begin{equation}\label{prod_prod}
\lim\limits_{\delta\to0}\mathrm{meas}\big\{s\in [0,t]\,:\,|B_s(\omega)-B_s(\omega_1)|<\delta\big\}=0
\end{equation}
$\mathbb P\times\mathbb P$-a.s.
Due to the condition ($H_{\rm um}$), for any pair $(\omega,\omega_1)$ such that (\ref{prod_prod}) is fulfilled, we have
$$
\lim\limits_{\eps\to 0}E\big\{\exp\big(Y_t^{\eps,x}(\omega)\big) \exp\big(Y_t^{\eps,x}(\omega_1)\big)\big\}
$$
$$
=\lim\limits_{\eps\to 0}E\big\{\exp\big(Y_t^{\eps,x}(\omega)\big)\big\}\,\lim\limits_{\eps\to 0}E \big\{\exp\big(Y_t^{\eps,x}(\omega_1)\big)\big\},
$$
and the desired statement follows from Proposition \ref{prop-tcl}.
\epf

\subsection{The case $\alpha=2\beta>0$}
This is the ``central case'', where $\alpha/4+\beta/2=\alpha/2$.
In this case, $\gamma=\beta=\alpha/2$, and we consider w. l. o. g. the case
where $\gamma=\beta=1$, $\alpha=2$. This means that we consider the PDE
\begin{equation}\label{eq-eps-cent}
\left\{
\begin{aligned}
\frac{\partial u^\eps}{\partial t}(t,x)&
=\frac{1}{2}\frac{\partial^2 u^\eps}{\partial x^2}(t,x)
+\frac{1}{\eps} c\left(\frac{t}{\eps^2},\frac{x}{\eps}\right)u^\eps(t,x),
\quad t\ge0,\, x\in\R;\\
u^\eps(0,x)&=g(x),\quad x\in\R,
\end{aligned}
\right.
\end{equation}
whose solution is given by the Feynman--Kac formula
\[
u^\eps(t,x)=\E\left[g(x+B_t)\exp\left(\eps^{-1}\int_0^t
c\left(\frac{s}{\eps^2},\frac{x+B_s}{\eps}\right)ds\right)\right].
\]

We will show that the limit of $u^\eps(t,x)$, as $\eps\to 0$, is a deterministic function.

Let us define
\[
Y^{\eps,x}_t=
\eps^{-1}\int_0^t
c\left(\frac{s}{\eps^2},\frac{x+B_s}{\eps}\right)ds,
\]
Then
\[
u^\eps(t,x)=\E\left[g(x+B_t)\exp(Y^{\eps,x}_t)\right].
\]
The random variable $Y^{\eps,x}_t$ is defined on the product
probability space $(\Ss\times\Omega,\A\otimes\F,P\times\P)$.

The limit of $u^\eps(t,x)$ will be obtained by a combination of Proposition
\ref{prop-tcl} (in the case $\beta=1/2$) and some uniform integrability property, which we now establish. Let us
prove the uniform in $\eps>0$ and $\omega\in\Omega$ integrability
with respect to the measure $P$ of the random variable
$$
\exp(Y_t^\eps)=\exp\bigg(\eps\int\limits_0^{t/\eps^2}
c\left(s,\frac{x}{\eps}+B_s\right)ds\bigg),\qquad
t>0.
$$
Because we need slightly different versions of the same result in other sections of this paper, we prove a more general result, which will be used in this section with $\nu=0$.

\begin{lemma}\label{l_uni_inte}
If the assumption $(H_{\rm um})$ is satisfied, then there exists $C$ such that for
all $\e>0$ and $\nu\in\R$,
\begin{equation}\label{ravnint}
E\exp\bigg(4\eps\int\limits_0^{t/\eps^2}
c(s,x\eps^{\frac{\nu}{2}-1}+B_{\e^\nu s})ds\bigg)\le C.
\end{equation}
\end{lemma}

\begin{remark} The condition  $\alpha(r)\le
C(1+r)^{-(1+\delta/2)}$, which is weaker than $(H_{\rm um})$, does imply that
$\rho\in L^1(\R_+)$.
However, the proof would be slightly more delicate. In particular, the parameter $\gamma$ which appears in the proof below should be choosen as a function of $\delta$.
\end{remark}

\bpf
Let $\gamma$ be an arbitrary positive number such that
$0<\gamma<1/2$, and consider an equidistant partition of the
interval $[0,\frac{t}{\eps^2}]$, the length of all subintervals
being equal to $\eps^{\gamma-1}$ (without loss of generality we
assume that $t\eps^{-(\gamma+1)}$ is an integer and, moreover, an
even number). We estimate separately the contribution of all the
subintervals with even numbers and of those with odd numbers. It
suffices to show that, with $\rho=\nu/2-1$,
\begin{equation}\label{vst_0}
E\exp\bigg(8\sum\limits_{j=1}^{t\eps^{-(\gamma+1)}/2}
\eps\int\limits_{2(j-1)\eps^{(\gamma-1)}}^{(2j-1)\eps^{(\gamma-1)}}
c(s,x\eps^\rho+B_{\e^\nu s})ds\bigg)\le C,
\end{equation}
$$
E\exp\bigg(8\sum\limits_{j=1}^{t\eps^{-(\gamma+1)}/2}
\eps\int\limits_{(2j-1)\eps^{(\gamma-1)}}^{2j\eps^{(\gamma-1)}}
c(s,x\eps^\rho+B_{\e^\nu s})ds\bigg)\le C.
$$
We introduce the notation
$$
\lambda_j^\eps=8\eps\int\limits_{2(j-1)\eps^{(\gamma-1)}}^{(2j-1)\eps^{(\gamma-1)}}
c(t,x\eps^\rho+B_{\e^\nu t})dt;\qquad \F_j^\eps=
\sigma\{c(t,x)\,:\,t\le2(j-1)\eps^{(\gamma-1)},\ x\in \R\}.
$$
Since $|c(s,x)|\le C$, we have the bound
\begin{equation}\label{vst_01}
|\lambda_j^\eps|\le c\eps^\gamma
\end{equation}
and, moreover,
\begin{equation}\label{vst_1}
E\exp(\lambda_j^\eps)=E\Big(1+\lambda_j^\eps+\frac{(\lambda_j^\eps)^2}{2!}+\dots+
\frac{(\lambda_j^\eps)^k}{k!}\Big)+\circ(\eps^{\gamma+1}),
\end{equation}
provided $k\ge(\frac{1}{\gamma}+1)$. The last term on the right hand
side admits the bound $|\circ(\eps^{\gamma+1})|\le
\kappa(\eps)\eps^{\gamma+1}$, where $\kappa$ is a deterministic function
defined on $\R_+$, which is such that $\kappa(\eps)\to0$, as $\eps\to0$.
 Since the random
field $\{c(t,x),\  t\ge0,x\in\R\}$ is centered,
$E\lambda_j^\eps=0$. Then
\begin{equation}\label{vst_2}
\begin{split}
E((\lambda_j^\eps)^2)&=64\eps^2
\int\limits_{2(j-1)\eps^{(\gamma-1)}}^{(2j-1)\eps^{(\gamma-1)}}
\int\limits_{2(j-1)\eps^{(\gamma-1)}}^{(2j-1)\eps^{(\gamma-1)}}
E\big(c(t,x\eps^\rho+B_{\e^\nu t})c(s,x\eps^\rho+B_{\e^\nu s})\big)dsdt\\
&\le c\eps^2\int\limits_0^{\eps^{(\gamma-1)}}
\int\limits_0^{\eps^{(\gamma-1)}}\rho(t-s)dtds\\
&\le
C\eps^{1+\gamma};
\end{split}
\end{equation}
For $m\ge 2$ we obtain
\begin{equation}\label{vst_3}
\begin{split}
|E(\lambda_j^\eps)^m|&\le c_m\eps^m
\int\limits_{2(j-1)\eps^{(\gamma-1)}}^{(2j-1)\eps^{(\gamma-1)}}\!\!\!\!\dots
\!\!\!\!\int\limits_{2(j-1)\eps^{(\gamma-1)}}^{(2j-1)\eps^{(\gamma-1)}}
|E\big(c(t_1,x\eps^\rho+B_{\e^\nu t_1})\dots c(t_m,x\eps^\rho+B_{\e^\nu t_m})\big)|dt_1\dots dt_m
\\
&= c_m\eps^m
\int\limits_{2(j-1)\eps^{(\gamma-1)}}^{(2j-1)\eps^{(\gamma-1)}}\!\!\!\!\dots
\!\!\!\!\int\limits_{2(j-1)\eps^{(\gamma-1)}}^{(2j-1)\eps^{(\gamma-1)}}
\Big|E\Big(c(t_1,x\eps^\rho+B_{\e^\nu t_1})\prod_{i=2}^m c(t_i,x\eps^\rho+B_{\e^\nu t_i})\Big)\Big|dt_1\dots dt_m
\\
&\le
c_m\eps^m\|c(\cdot,\cdot)\|^m_\infty\int\limits_0^{\eps^{(\gamma-1)}}\dots
\int\limits_0^{\eps^{(\gamma-1)}}\rho(\min\limits_{2\le i\le
m}|t_i-t_1|)dt_1\dots dt_m
\\
&\le \sum\limits_{i=2}^m
c_m\eps^m\int\limits_0^{\eps^{(\gamma-1)}}\!\!\dots\!\!
\int\limits_0^{\eps^{(\gamma-1)}}\!\!\!\rho(|t_i-t_1|)dt_1\dots
dt_m\\
 &\le c_m\eps^m
\eps^{(\gamma-1)(m-1)}
=c_m\eps^{(1+(m-1)\gamma)}.
\end{split}
\end{equation}
Combining (\ref{vst_1})--(\ref{vst_3}) together gives
\begin{equation}\label{vst_4}
E\exp(\lambda_j^\eps)\le 1+c\eps^{(\gamma+1)}.
\end{equation}
Now, letting $L=t/(2\eps^{\gamma+1})$, we can estimate the left
hand side of (\ref{vst_0}) as follows

\begin{align*}
E\exp\big(\sum\limits_{j=1}^L \lambda_j^\eps\big)&=
E\bigg(E\big\{\exp\big(\sum\limits_{j=0}^L
\lambda_j^\eps\big)\,\big|\,
\F_{L-1}^\eps\big\}\bigg)\\
&=E\Big[\exp\Big(\sum\limits_{j=0}^{L-1} \lambda_j^\eps\Big)
E\big\{\exp\big(\lambda_L^\eps\big)\,\big|\,
\F_{L-1}^\eps\big\}\Big]\\
&=E\exp\Big(\sum\limits_{j=0}^{L-1}
\lambda_j^\eps\Big) E\big(\exp(\lambda_L^\eps)\big)\\
&\quad+
E\Big[\exp\Big(\sum\limits_{j=0}^{L-1} \lambda_j^\eps\Big)
E\big\{\big[\exp\big(\lambda_L^\eps\big)-
E\exp\big(\lambda_L^\eps\big)\big]\,\big|\,
\F_{L-1}^\eps\big\}\Big]\\
&\le(1+c\eps^{(1+\gamma)})
E\exp\Big(\sum\limits_{j=0}^{L-1} \lambda_j^\eps\Big)\\
&\quad+E\Big[\exp\Big(\sum\limits_{j=0}^{L-1} \lambda_j^\eps\Big)
E\big\{\big[\exp\big(\lambda_L^\eps\big)-
E\exp\big(\lambda_L^\eps\big)\big]\,\big|\, \F_{L-1}^\eps\big\}\Big]
\end{align*}
Using successfully Proposition 7.2.6 from \cite{EtKu}, the obvious
inequality
$$\|\exp(\xi)-E\exp(\xi)\|_\infty\le\|\exp(\xi)\|_\infty\|\xi\|_\infty,$$
the bound (\ref{vst_01}), and the fact that $\gamma<1/2$, we
obtain the inequality
\begin{align*}
\big|E\big\{\big[\exp\big(\lambda_L^\eps\big)-
E\exp\big(\lambda_L^\eps\big)\big]\,\big|\, \F_{L-1}^\eps\big\}\big|
&\le c \alpha(\eps^{(\gamma-1)})\|\exp\big(\lambda_L^\eps\big)-
E\exp\big(\lambda_L^\eps\big)\|_{\lower 4pt\hbox{$\scriptstyle
L^\infty$}}\\
&\le c\|\exp\big(\lambda_L^\eps\big)\|_{\lower 4pt\hbox{$\scriptstyle
L^\infty$}}\|\lambda_L^\eps\|_{\lower 4pt\hbox{$\scriptstyle
L^\infty$}}\alpha(\eps^{(\gamma-1)})\\
&\le c\eps^\gamma
(\eps^{(\gamma-1)})^{-(2+\delta)}\\&=c\eps^{(2-\gamma+\delta-\gamma
\delta)}\\&\le c\eps^{(\gamma+1)}.
\end{align*}
Finally, we conclude that
$$
E\exp\big(\sum\limits_{j=0}^L \lambda_j^\eps\big)\le
E\exp\big(\sum\limits_{j=0}^{L-1}
\lambda_j^\eps\big)(1+c\eps^{(\gamma+1)}).
$$
Iterating this inequality, we get after $L$ steps:
$$
E\exp\big(\sum\limits_{j=0}^L \lambda_j^\eps\big)\le
\big(1+c\eps^{(\gamma+1)}\big)^L\le
\big(1+c\eps^{(\gamma+1)}\big)^{(t/\eps^{(\gamma+1)})}\le
\exp(2ct).
$$
The contribution of the odd terms can be estimated exactly in the
same way, and the proof is complete.
\epf

\begin{proposition}\label{aeq2b}

We have that
$$
E\Big(\big(\E(g(x+B_t)[e^{Y_t^{\eps,x}}-e^{t\Sigma}])\big)^2\Big)\,\to\,0
$$
as $\eps\to 0$, where
$$\Sigma=\int_0^\infty \E\Phi(r,B_r)dr.$$
\end{proposition}
\bpf
We have to compute
\begin{align}\label{q-ku}
E&\Big(\big(\E(g(x+B_t)[e^{Y_t^{\eps,x}}-e^{t\Sigma}])\big)^2\Big)\\
\nonumber
&=\int_\Omega\int_\Omega g(x+B_t(\omega))g(x+B_t(\omega'))\,
E\big(e^{Y_t^{\eps,x}(\omega)+Y_t^{\eps,x}(\omega')}\big)\P(d\omega)
\P(d\omega')\\
\nonumber
&-2e^{t\Sigma}\E g(x+B_t)\int_\Omega g(x+B_t(\omega))
Ee^{Y_t^{\eps,x}(\omega)}\P(d\omega)+(\E g(x+B_t))^2e^{2t\Sigma}.
\end{align}
It follows from Proposition \ref{prop-tcl}, Proposition \ref{prop-tcl-prod} and Lemma \ref{l_uni_inte} that
\begin{equation}\label{fo52}
Ee^{Y_t^{\eps,x}}\,\to\,e^{t\Sigma},
\end{equation}
in $\P$--probability
as $\eps\to 0$, and
\begin{equation}\label{fo_53}
Ee^{Y_t^{\eps,x}(\omega)+Y_t^{\eps,x}(\omega')}\,\to\,e^{2t\Sigma}
\end{equation}
in
$\P(d\omega)\times\P(d\omega')$--probability as $\eps\to0$.
Passing to the limit, as $\eps\to0$, on the right-hand side of (\ref{q-ku})
we arrive at the required assertion.
%
%
%
\epf

An immediate consequence of the last Proposition is the

\begin{corollary}
The limit $u$ of $u^\eps$ is given by
\[
u(t,x)=\E[g(x+B_t)]\exp\left(t\Sigma\right),
\]
which is a solution of the deterministic parabolic PDE
\begin{equation}
\left\{
\begin{aligned}
\frac{\partial u}{\partial t}(t,x)&=\frac{1}{2}
\frac{\partial^2 u}{\partial x^2}(t,x)+\Sigma u(t,x),\quad t\ge 0,\
x\in\R;
\\
u(0,x)&=g(x),\ x\in\R.
\end{aligned}
\right.
\end{equation}
\end{corollary}

\subsection{The case $0<2\beta<\alpha$}
Without loss of generality we choose
$\alpha=1$ and $0<\beta<1/2$. Hence $\gamma=1$. We know from
Proposition \ref{prop-tcl} that
$$
Y_t^{\eps,x}=\frac{1}{\sqrt{\eps}}\int_0^t c\Big(\frac{s}{\eps},
\frac{x+B_s}{\eps^\beta}\Big)ds,
$$
converges, as $\eps\to 0$, in $\P$--probability weakly under $P$ to
the Gaussian law $N(0,t\int_\R\Phi(u,0)du)$.

  We now note that the r. v. $Y^{\eps,x}_t$ can be rewritten as
$$Y^{\eps,x}_t=\sqrt{\eps}\int_0^{t/\eps}c\left(s,x\eps^{-\beta}
+B_{s\eps^{(1-\beta)}}\right)ds.$$
Hence it follows from Lemma \ref{l_uni_inte} with $\nu=(1-\beta)$ that
$$\sup_{\eps>0}E\left(\exp[4Y^{\eps,x}_t]\right)\le C.$$
Consequently, by the same arguments as those in the previous section,
we can show the
\begin{proposition}
The limit $u$ of $u^\eps$ is given by
\[
u(t,x)=\E[g(x+B_t)]\exp\left(t\Sigma'\right),
\]
which is a solution of the deterministic parabolic PDE
\begin{equation}
\left\{
\begin{aligned}
\frac{\partial u}{\partial t}(t,x)&=\frac{1}{2}
\frac{\partial^2 u}{\partial x^2}(t,x)+\Sigma' u(t,x),\quad t\ge 0,\
x\in\R;
\\
u(0,x)&=g(x),\ x\in\R,
\end{aligned}
\right.
\end{equation}
where $\Sigma'=\int_0^\infty \Phi(u,0)du$.
\end{proposition}


\subsection{The case $\beta=0$}

In this case, $\gamma=\alpha/2$. Without loss of generality, we restrict ourselves to the case $\alpha=1$.

We will study the limit behaviour of $u^\eps$ under the following additional
assumption:
\begin{itemize}
\item[($H\!\hbox{\it \"o}$)]
For each $s\in\R$ the realizations $c(s,y)$ are a.s. H\"older continuous
in $y\in\R$ with a deterministic exponent $\theta>0$.
Moreover,
$$
|c(s,y_1)-c(s,y_2)|\le c|y_1-y_2|^\theta,
$$
with a deterministic constant $c$.
\item[($H^{q}_{\rm um}$)] There is $\delta>0$ such that
$$
\alpha_{\rm um}(r)\le C(1+r)^{-(q+\delta)},
$$
where $q(\theta)=3$ if $\theta>1/3$ and $q(\theta)=(k+1)$ if $\theta\in[1/k, 1/(k+1))$, $k\ge3$, $k\in\mathbb N$.
\end{itemize}
Proposition \ref{prop-tcl} still applies here. However, it is not sufficiently precise to
be useful in this case. The reason is that the limit of $u^\eps$ will not be deterministic in this case.
Convergence will be only in law, not in probability or in mean square. Going back to the proof of Proposition
\ref{aeq2b}, which is not valid in the present case, we note that while the limiting law of
$Y^{\eps,x}(\omega)$ is the same as above, that of $(Y^{\eps,x}(\omega),Y^{\eps,x}(\omega'))$
will be dramatically different.

Consider the exponent in the above Feynman--Kac formula, written
in its first form. It reads
\[
Y^{x,\e}_t=\frac{1}{\sqrt{\e}}\int_0^tc\left(\frac{s}{\eps},x+B_s\right)ds
=\int_0^t W^\eps(ds,x+B_s),
\]
where
\[
W^\eps(t,x):=\frac{1}{\sqrt{\eps}}\int_0^t
c\left(\frac{s}{\eps},x\right)ds.
\]
We have the
\begin{proposition}\label{last_CLT} Under assumptions ($H\!\hbox{\it \"o}$) and ($H^{q(\theta)}_{\rm um}$), as $\eps\to0$,
$$
W^\eps\to W
$$
in $P$--law, as random elements of\ \ $C(\R_+\times\R)$,
where $\{W(t,x),\, t\ge0,\, x\in\R\}$ is a centered Gaussian process with covariance
function given by
\[
E(W(t,x)W(t',x'))=t\wedge t'\times R(x-x'),
\]
with
\[
R(x)=\int_\R\Phi(r,x)dr.
\]
\end{proposition}
\bpf
The convergence of finite dimensional distributions is a direct consequence of the functional Central Limit Theorem for stationary processes having good enough mixing properties. Namely, according to the statements in \cite{Bi}, Chapter 4, \S 20, under the assumption ($H^{q}_{\rm um}$) with $q\ge 1$, for any finite set
$x^1,x^2,\dots,x^m$ the family
$$
\{W^\eps(\cdot,x^1),\dots,W^\eps(\cdot,x^m)\}
$$
converges in law, as $\eps\to0$, in the space $(C(0,T))^m$, towards a $m$-dimensional Wiener process with covariance matrix
$$
\sigma_{ij}=\int\limits_0^\infty E\big(c(s,x^i)c(0,x^j)+c(s,x^j)c(0,x^i)\big)ds= \int\limits_{-\infty}^\infty\Phi(s,x^i-x^j)ds=R(x^i-x^j).
$$
The desired result will follow if we prove the tightness of $\{W^\eps,\eps>0\}$ in $C(\R_+\times\R)$. In order to prove that this family is tight it suffices to show that
there are two numbers $\nu_1>0$ and $\nu_2>2$ such that
$$
E|W^\eps(s_1,y_1)-W^\eps(s_2,y_2)|^{\nu_1}\le C(|s_1-s_2|^{\nu_2}+|y_1-y_2|^{\nu_2})
$$
with a constant $C$ which does not depend on $\eps$.
For presentation simplicity we consider the case $\theta>1/3$ and $q=3$; other cases can be studied exactly in the same way.

We have
$$
E\bigg(\eps^{-1/2}\int\limits_{s_1}^{s_2}c\Big(\frac{t}{\eps},y\Big)dt\bigg)^6
=\eps^{-3}\int\limits_{s_1}^{s_2}\dots \int\limits_{s_1}^{s_2}E \Big\{c\Big(\frac{t_1}{\eps},y\Big)c\Big(\frac{t_2}{\eps},y\Big)\dots c\Big(\frac{t_6}{\eps},y\Big)\Big\}dt_1\dots dt_6
$$
$$
=\eps^3\int\limits_{s_1/\eps}^{s_2/\eps}\dots\int\limits_{s_1/\eps}^{s_2/\eps} E\{c(t_1,y)\dots c(t_6,y)\}dt_1\dots dt_6.
$$
Let us now introduce the set
$$
S(r)=\left\{(t_1,\dots,t_6)\in \Big[\frac{s_1}{\eps},\frac{s_2}{\eps}\Big]^6\,:\, \max\limits_{1\le i\le 6} \min\limits_{j\not=i}|t_i-t_j|\le r \right\}.
$$
It is an easy exercise to check that
$$
V(r)=\mathrm{Vol}(S(r))\le Cr^3\frac{(s_2-s_1)^3}{\eps^3}.
$$
If for some $i\in\{1,\dots,6\}$ it holds $|t_i-t_j|\ge r$ for all $j\not=i$
(without loss of generality $i=1$), then, taking into account ($H^q_{\rm um}$), we have
$$
|E\{c(t_1,y)c(t_2,y)\dots c(t_6,y)\}|=|E(c(t_2,y)\dots c(t_6,y)E\{c(t_1,y)|{\cal F}_{\{t_2,\dots,t_6\}}\})|\le
$$
$$
\le C
E(|c(t_2,y)\dots c(t_6,y)|)(1+r)^{-(3+\delta)}\|c(t_1,y)\|_{L^\infty({\cal A})}
\le (1+r)^{-(3+\delta)}\|c\|^6_{L^\infty({\cal A})}\le
$$
$$
\le
C(1+r)^{-(3+\delta)}.
$$
Therefore,
$$
\eps^3\int\limits_{s_1/\eps}^{s_2/\eps}\dots\int\limits_{s_1/\eps}^{s_2/\eps} E\{c(t_1,y)\dots c(t_6,y)\}dt_1\dots dt_6\le C\eps^3\int\limits_0^{\sqrt{6}(s_2-s_1)/\eps}\frac{dV(r)}{(1+r)^{3+\delta}}\le
$$
$$
C\eps^3\int\limits_0^\infty\frac{dV(r)}{(1+r)^{3+\delta}}\
=C\eps^3\big(V(r) (1+r)^{-(3+\delta)}\big)\big|_0^\infty+(3+\delta)C\eps^3 \int\limits_0^\infty\frac{V(r)dr}{(1+r)^{(4+\delta)}}\le
$$
$$
\le C(s_2-s_1)^3\int\limits_0^\infty\frac{r^3\,dr}{(1+r)^{(4+\delta)}}\le C(s_2-s_1)^3.
$$

Similarly, by ($H\!\hbox{\it \"o}$) and ($H^q_{\rm um}$) one has
$$
E\bigg(\eps^{-1/2}\int\limits_0^s\Big(c\Big(\frac{t}{\eps},y_1\Big)- c\Big(\frac{t}{\eps},y_2\Big) dt\bigg)^6
=
$$
$$
\eps^3\int\limits_0^{s/\eps}\dots \int\limits_0^{s/\eps} E\{(c(t_1,y_1)-c(t_1,y_2))\dots (c(t_6,y_1)-c(t_6,y_2))\}dt_1\dots dt_6\le
$$
$$
C\eps^3|y_1-y_2|^{6\theta}\int\limits_0^{\sqrt{6}T/\eps}\frac{dV_T(r)}{(1+r)^{3+\delta}}.
$$
where $V_T(r)$ stands for the volume of the set
$$
S_T(r)=\left\{(t_1,\dots,t_6)\in \Big[0,\frac{T}{\eps}\Big]^6\,:\, \max\limits_{1\le i\le 6} \min\limits_{j\not=i}|t_i-t_j|\le r \right\}.
$$
Straightforward computations show that
$$
V_T(r)\le Cr^3\frac{T^3}{\eps^3},
$$
This yields
$$
E\bigg(\eps^{-1/2}\int\limits_0^s\Big(c\Big(\frac{t}{\eps},y_1\Big)- c\Big(\frac{t}{\eps},y_2\Big) dt\bigg)^6\le CT^3|y_1-y_2|^{6\theta} \int\limits_0^\infty \frac{r^3dr}{(1+r)^{4+\delta}}
$$
Since $\theta>1/3$, this implies the desired estimate.
\epf

The statement of the last proposition remains valid if we replace assumption ($H\!\hbox{\it \"o}$) with the following one
\begin{equation}\label{ho_mo1}
\|c(s,\cdot)\|_{C^\theta([0,1])}\le C(s,\omega)
\end{equation}
with $C(s,\omega)$ such that
\begin{equation}\label{ho_mo2}
E(|C(\cdot)|^{2(2q(\theta)-1)})<\infty;
\end{equation}
by the stationarity the law of $C(s,\omega)$ does not depend on $s$.
In this case the exponent in ($H_{\rm um}^q$) is to be chosen as follows
\begin{equation}\label{ho_mo3}
\alpha_{\rm um}(r)\le C(1+r)^{-(\tilde q+\delta)},
\qquad\tilde q(\theta)=2q(\theta).
\end{equation}

\begin{proposition}\label{last_CLTbis}
Let assumptions (\ref{ho_mo1})-(\ref{ho_mo3}) be fulfilled. Then the statement of Proposition \ref{last_CLT} holds true.
\end{proposition}
\bpf
Again for definiteness we assume that $\theta>1/3$, other cases can be treated similarly.  Then $q(\theta)=3$, $\tilde q(\theta)=6$ and $2(2q(\theta)-1)=10$. Without loss of generality we may assume that
$$
\max\limits_{2\le j\le 6}|t_1-t_j|= \max\limits_{1\le i,j\le 6}|t_i-t_j.|
$$
If  $\max\limits_{2\le j\le 6}|t_1-t_j|\ge r$, then by Lemma VIII.3.102 in \cite{JaSh} with $p=2$ and $q=2$, for any $y_1$ and $y_2$ such that $|y_1-y_2|\le 1$, it holds
$$
|E\{(c(t_1,y_1)-c(t_1,y_2))\dots (c(t_6,y_1)-c(t_6,y_2))\}|\le
$$
$$
\le \frac{C}{(1+r)^{(3+\delta/2)}}\|c(t_1,y_1)-c(t_1,y_2)\|_{L^2({\cal S})}\,
\|(c(t_2,y_1)-c(t_2,y_2))\dots(c(t_6,y_1)-c(t_6,y_2))\|_{L^2({\cal S})}
$$
$$
\le \frac{C|y_1-y_2|^{6\theta}}{(1+r)^{(3+\delta/2)}}
\|C(t_1))\|_{L^2(\Omega)}\prod\limits_{j=2}^6\|C(t_j))\|_{L^{10}(\Omega)}
\le C\frac{|y_1-y_2|^{6\theta}}{(1+r)^{(3+\delta/2)}};
$$
the H\"older inequality and the stationarity of $C(s)$ has also been used here.
This yields
$$
E\bigg(\eps^{-1/2}\int\limits_0^s\Big(c\Big(\frac{t}{\eps},y_1\Big)- c\Big(\frac{t}{\eps},y_2\Big) dt\bigg)^6
=
$$
$$
\eps^3\int\limits_0^{s/\eps}\dots \int\limits_0^{s/\eps} E\{(c(t_1,y_1)-c(t_1,y_2))\dots (c(t_6,y_1)-c(t_6,y_2))\}dt_1\dots dt_6\le
$$
$$
C|y_1-y_2|^{6\theta}\int\limits_0^\infty\frac{r^3\,dr} {(1+r)^{4+(\delta/2)}}\le C|y_1-y_2|^{6\theta}.
$$
The estimate
$$
E\bigg(\eps^{-1/2}\int\limits_{s_1}^{s_2}c\Big(\frac{t}{\eps},y_1\Big) dt\bigg)^6\le C|s_2-s_1|^3
$$
can be proved in the same way, and the desired statement follows.
\epf


As we shall see below, the exponent in the Feynman--Kac formula converges towards
\[
\int_0^t W(ds,x+B_s)=\int_0^t\int_\R W(ds,y)L(s,y-x)dy,
\]
where again $L(t,z)$ stands for the local time of the process $B$ at time $t$ and location $z$.

Let us note that the left hand--side of the last identity can be defined without any reference to local time. Recall that $W$ and $B$ are independent, hence it
suffices to define the stochastic integral
\[
\int_0^t W(ds,f(s)),\quad t\ge 0,
\]
with $f\in C(\R_+)$.
\begin{proposition}
To any $f\in C(\R_+)$, we associate the
 continuous centered Gaussian process
\[
\{Y_t:=\int_0^t W(ds,f(s)),\: t\ge 0\}
\]
with the covariance function $(t\wedge t')\,R(0)$, which is, for
each $t>0$, the limit in probability  as $n\to\infty$ of the
sequence
\[
Y^n_t:=\sum_{k=1}^{[t2^n]}\left[W(k2^{-n},f(k2^{-n}))-W((k-1)2^{-n},f(k2^{-n}))\right].
\]
\end{proposition}
\bpf  The fact that $\{Y^n_t,\: n\ge1\}$ is a Cauchy sequence in $L^2(P)$
will follow from the fact that $E(Y^n_tY^m_t)$ converges to a finite limit as
$n$ and $m$ tend to infinity. This is indeed the case, since for $n>m$,
\begin{align*}
E(Y^n_tY^m_t)&=[t2^{-n}]2^n\sum_{\ell=1}^{[t2^m]}\sum_{k=(\ell-1)2^{n-m}}^{\ell2^{n-m}}
R(f(k2^{-n})-f(\ell2^{-m}))\\
&\to t R(0),
\end{align*}
as $n$ and $m$ tend to infinity, with $n>m$. The fact that $Y_t$ is Gaussian and centered
follows easily, as well as the formula for the covariance.
\epf

Note that the conditional law of $Y_t=\int_0^tW(ds,x+B_s)$, given $\{B_s,\: 0\le s\le t\}$
is the law $N(0,tR(0))$. It does not depend on the realization of
$\{B_s,\: 0\le s\le t\}$, in agreement with Proposition \ref{prop-tcl}. However, $Y_t$ does depend on
$\{B_s,\: 0\le s\le t\}$.
This follows in particular from the fact that if $B$ and $B'$ are two trajectories of the
Brownian motion,
\[
E\left[\int_0^tW(ds,B_s)\int_0^{t'}W(ds,B'_s)\right]=
\int_0^{t\wedge t'}R(B_s-B'_s)ds.
\]

The uniform integrability here is easy to establish. Indeed, we saw in the previous section that it is sufficient to prove  that the collection
of r. v.
$$\left\{\E\exp(2Y^{x,\e}_t),\ \e>0\right\}$$
is $P$--tight. Since those are non--negative random variables, a
sufficient condition is that
$$\sup_{\e>0}E\E\exp(2Y^{x,\e}_t)<\infty,$$
and we can very well interchange the order of expectation.
Now Lemma \ref{l_uni_inte} above,  in the case
$\nu=2$, implies that
$$\sup_{\e>0}E\left(\exp\left[\frac{1}{\sqrt{\e}}\int_0^tc\left(\frac{s}{\eps},x+B_s\right)ds\right]\right)\le C,$$
where $C$ is a finite constant. This is easily seen by making the following change of variable~:
$$\frac{1}{\sqrt{\e}}\int_0^tc\left(\frac{s}{\eps},x+B_s\right)ds=
\eta\int_0^{t/\eta^2}c(r,x+B_{\eps^2r})dr, $$
with $\eta=\sqrt{\eps}$.

We can finally establish the
\begin{theorem}
Under assumptions ($H\!\hbox{\it \"o}$) and ($H_{\rm um}^{q(\theta)}$) for each $(t,x)\in\R_+\times\R$,
$$u^\e(t,x)\to u(t,x):=\E\left[g(X^x_t)\exp\left(\int_0^t\int_\R W(ds,y)L(s,y-x)dy\right)\right]$$
in $P$--law, as $\e\to0$.
\end{theorem}
\bpf
Note that
\begin{align*}
Y^{x,\e}_t&=\int_0^t\int_\R W^\e(ds,y)L(s,y-x)dy\\
&=\int_\R W^\e(t,y)L(t,y-x)dy-\int_0^t\int_\R W^\e(s,y)L(ds,y-x)dy
\end{align*}
Define the functional $\Psi_{t,x} : [0,t]\times\R\to \R$ as
$$\Psi_{t,x}(\varphi):=\E\left[g(X^x_t)\exp\left(\int_\R \varphi(t,y)L(t,y-x)dy-\int_0^t\int_\R \varphi(s,y)L(ds,y-x)
dy\right)\right].$$
All we have to show is that
$$u^\e(t,x)=\Psi_{t,x}(W^\e)\to\Psi_{t,x}(W)$$
in $P$--law, which follows from Proposition \ref{last_CLT}
and uniform integrability, since $\Psi_{t,x}$ is continuous.
\epf

The corresponding limiting SPDE reads (in Stratonovich form)
\begin{equation*}
\left\{
\begin{aligned}
d u(t,x)&
=\frac{1}{2}\frac{\partial^2 u}{\partial x^2}(t,x)dt
+u(t,x)\circ W(dt,x),
\quad t\ge0,\, x\in\R;\\
u(0,x)&=g(x),\quad x\in\R.
\end{aligned}
\right.
\end{equation*}
We can rewrite this SPDE in It\^o form as follows
\begin{equation*}
\left\{
\begin{aligned}
d u(t,x)&
=\frac{1}{2}\frac{\partial^2 u}{\partial x^2}(t,x)dt+\frac{1}{2}u(t,x)R(0)dt
+u(t,x) W(dt,x),
\quad t\ge0,\, x\in\R;\\
u(0,x)&=g(x),\quad x\in\R.
\end{aligned}
\right.
\end{equation*}

\end{document}